\documentclass[12pt]{amsart}
\usepackage{amsmath}
\usepackage{amssymb}
\usepackage{amstext}
\usepackage{amscd}
\usepackage[matrix,arrow,ps]{xy}

\newtheorem{teo}{Theorem}[section]
\newtheorem{prop}[teo]{Proposition}
\newtheorem{lem}[teo]{Lemma}
\newtheorem{cor}[teo]{Corollary}
\newtheorem{conj}[teo]{Conjecture}

\newtheorem{defini}[teo]{Definition}

\newtheorem{rem}[teo]{Remark}

\newcommand{\End}{{\rm End}}
\newcommand{\GSp}{{\rm GSp}}
\newcommand{\GL}{{\rm GL}}

\newcommand{\Sh}{{\rm Sh}}

\newcommand{\Gal}{{\rm Gal}}
\newcommand{\Res}{{\rm Res}}

\newcommand{\der}{{\rm der}}
\newcommand{\ab}{{\rm ab}}
\newcommand{\ad}{{\rm ad}}

\newcommand{\CC}{{\mathbb C}}
\newcommand{\RR}{{\mathbb R}}
\newcommand{\ZZ}{{\mathbb Z}}
\newcommand{\QQ}{{\mathbb Q}}
\newcommand{\NN}{{\mathbb N}}

\newcommand{\HH}{{\mathbb H}}

\newcommand{\GG}{{\mathbb G}}
\newcommand{\SSS}{{\mathbb S}}
\newcommand{\AAA}{{\mathbb A}}
\newcommand{\UUU}{{\mathbb U}}
\newcommand{\lto}{\longrightarrow}

\newcommand{\cA}{{\mathcal A}}

\newcommand{\cL}{{\mathcal L}}

\newcommand{\ol}{\overline}

\newcommand{\oL}{\overline{L}}

\newtheorem{defi}{Definition}[section]

\newcommand{\wt}{\widetilde}

\title{Mumford-Tate and generalised Shafarevich conjectures.}

\author{Emmanuel Ullmo}
\address{Universit\'e de Paris-Sud, Bat 425 et Institut Universitaire de France, 91405, Orsay
Cedex France}
\email{ullmo@math.u-psud.fr}
\author{Andrei Yafaev}
\address{University
College London, Department of Mathematics, 25 Gordon street, WC1H
OAH, London, United Kingdom}
\email{yafaev@math.ucl.ac.uk}
\thanks{The research of the second author was supported by the Leverhulme Trust}

\begin{document}
\maketitle
\begin{abstract}
Dans cet article on formule les conjectures d'Isogenie, de Mumford-Tate et de Shafarevich
g\'en\'eralis\'ees et on d\'emontre qu'elles sont \'equivalentes.
\end{abstract}
\begin{abstract}
In this paper we formulate generalised Isogeny, Mumford-Tate and Shafarevich conjectures
and prove that they are equivalent.
\end{abstract}
\tableofcontents

\section{Introduction.}

Let $A$ be an abelian variety of dimension $g$ over a number field $L$.
For any prime number $l$, let $T_l(A)$ be the projective limit of
$A[l^n](\ol L)$, the $l^n$ torsion of $A$.
It is a free $\ZZ_l$-module of rank $2g$, classically called the \emph{Tate module} of $A$. 
We denote by $V_l(A)$
the $\QQ_l$-vector space $\QQ_l \otimes T_l(A)$.
The Tate module $T_l(A)$ has a continuous action by the Galois group $\Gal(\ol L/L)$.
In his landmark paper \cite{Fa}, Faltings proves the following statements:
\begin{enumerate}
\item The action of $\Gal(\ol L/L)$ on $V_l(A)$ is semi-simple.
\item The natural map
$$
\End_L(A)\otimes \ZZ_{l}\lto \End_{\Gal(\ol L/L)}(T_l(A))
$$
is an isomorphism.
\item There are only finitely many isomorphism classes of
abelian varieties defined over $L$ in the isogeny class of $A$.
\end{enumerate}
The second statement is known as the Tate isogeny conjecture (for abelian varieties). Tate proved these 
statements for abelian varieties over finite fields (see \cite{Tate}) and Faltings proved them
for abelian varieties over number fields using completely different techniques.
The second statement is a very special case of the general Tate conjecture on algebraic cycles on 
abelian varieties (see \cite{Tate2}, also \cite{Milne2}). 

The last statement is a consequence
of the  Shafarevich conjecture and constitutes an important step of its proof by Faltings.
We will therefore call this last statement the Shafarevich conjecture as well.
The classical statement of the Shafarevich conjecture
is that for a given number field $L$, integers $g$ and $d$ and a
finite set of places $S$ of $L$, there are only finitely many isomorphism classes of
abelian varieties over $L$ of dimension $g$ with a polarisation of degree $d^2$  
having good reduction outside  $S$. The statement 3 above is a direct
consequence of this.  

The action of $\Gal(\ol L/L)$ on $V_l(A)$ gives rise to a representation
$$
\rho_{A,l} \colon \Gal(\ol L/L) \lto \GL(V_l(A)).
$$

The $l$-adic monodromy group associated to $\rho_{A,l}$ is defined as the
Zariski closure $H_{l}$ of $\rho_{A,l} (\Gal(\ol L/L))$ in $\GL(V_{l}(A))$.
Let $M$ be the Mumford-Tate group of $A$.
It is known that the connected component $H_{l}^0$ 
of $H_{l}$ is independent of the field of definition $L$ of $A$ (see \cite{Se}) and  is contained in $M_{\QQ_l}$
 (\cite{Deligne},I, Prop. 6.2 and \cite{Borovoi}) .
The Mumford-Tate conjecture asserts that $H_{l}^0=M_{\QQ_l}$.
This conjecture is in general open. There are some easy implications between these conjectures.
Tate's method (see \cite{Tate}) shows that the Shafarevich conjecture implies the Tate isogeny conjecture. It is easy to see that
the Mumford-Tate conjecture implies the Tate isogeny conjecture. 
An attempt to prove the reverse inclusion leads to an idea for generalisation of the Tate isogeny conjecture
as explained below.
In this paper we formulate suitable generalisations of the  Mumford-Tate, Tate isogeny and Shafarevich conjectures and show that the three conjectures are equivalent.

We start by constructing a Galois representation attached to a $\ol\QQ$-valued point on a Shimura variety
\emph{ and } a representation of the group intervening in the definition of the Shimura variety 
in question. Our construction is to be compared to 
that of Shimura (see \cite{Shimura}) and Borovoi (see \cite{Borovoi}). 

When the Shimura
variety is the moduli space of principally polarised abelian varieties of dimension $g$
and the representation is the standard representation of the symplectic group,
our definition gives the $l$-adic representations on the Tate module
of an abelian variety discussed previously.

We then  formulate the \emph{generalised} Tate isogeny conjecture (that we simply call Isogeny conjecture),  
and \emph{generalised}
the Shafarevich conjecture in this context.
We show that they are both equivalent to the \emph{generalised} Mumford-Tate conjecture.
For Shimura varieties of Hodge type the generalised Mumford-Tate conjecture
is just the usual one for abelian varieties. 

The idea of the generalisation is to make the conjectures independent of 
the choice of the representation of the group defining the Shimura variety.
Tate isogeny conjecture is just our Isogeny conjecture for the standard symplectic representation. 
Known
cases of the Mumford-Tate conjecture provide some cases of these generalised
conjectures. 
We hope that our approach yields a new perspective on these conjectures. 

In particular, the Isogeny and Mumford-Tate conjectures are of `motivic' nature while the generalised Shafarevich conjecture 
is of arithmetic nature and we hope that the methods of Faltings may be adapted and yield an approach to this conjecture.

The implication `the generalised Shafarevich implies the Isogeny conjecture' is classical and is 
an adaptation of the arguments of Tate (\cite{Tate}). The implication `generalised Mumford-Tate implies Isogeny conjecture' is straightforward. The converse is easy but really requires the independence of the choice of the representation.
The implication `the generalised Mumford-Tate implies the generalised Shafarevich' is intuitive. Mumford-Tate 
conjecture states loosely speaking that the image of the Galois group is `big' and hence the fields of definition of points
lying in a generalised Hecke orbit should grow. However, technically this is the hardest part 
of the paper and requires a difficult result in group/measure theory which, we believe, is of independent interest.
 
As a sub-product of our construction of Galois representations we recover some previously known 
results on Galois representations attached to abelian varieties such as the fact
the image of $\rho_{A,l}$ is contained in $M(\QQ_l)$ (with the above notation)
and that the centre of $M$ is contained in the centre of the $l$-adic monodromy group
$H_{l}$ and coincides with it for Shimura varieties of Hodge type. Notice that 
classically such results are derived from hard theorems such as Deligne's theorem on absolute
Hodge cycles (see \cite{Deligne}) while in this paper we show that they follow from our
construction in a straightforward way.

\subsection{Notations and conventions.}

The algebraic closure of a field $k$ is denoted $\overline{k}$.

Number fields are always assumed to be given as
subfields of $\CC$. With this convention, if $Z$ is a an 
algebraic variety over a number field $L$ then the set
$Z(L)$ of $L$-rational points of $Z$ is a subset of the set of
the complex points $Z(\CC)$ of $Z$.

For a linear algebraic group $G$, $G^0$ denotes the identity component
of  $G$ for the Zariski topology and  $G^{ad}$ (adjoint group),
$G^{ab}$ (maximal abelian quotient) and $G^{der}$ (derived group)
have the usual meaning. The centre of $G$ is denoted by $Z_{G}$.
If $H$ is an algebraic subgroup of $G$ then $N_{G}(H)$
and $Z_{G}(H)$ are respectively the normaliser and the centraliser of $H$
in $G$. Reductive groups are assumed to be connected.

Let $V$ be a $k$-vector space and $A$ be a subset of $\End(V)$, then
$\End_{A}(V)$ denotes the endomorphism algebra of $V$ commuting
with the elements of $A$.

We write  $\AAA$ (resp. $\AAA_{f}$) for the ring  of ad\`eles (resp. finite ad\`eles) of $\QQ$.
A superscript $^l$ denotes a structure ``away from $l$''; a subscript
$_{l}$ denotes a structure ``at $l$''. Let $G$ be an algebraic group over $\QQ$,
we refer to \cite{Pi2} 0.6 for the definition
of a neat subgroup of $G(\AAA_{f})$.

Let $\SSS:=\Res_{\CC/\RR}\GG_{m,\CC}$ be the Deligne torus.
By a Shimura datum $(G,X)$ we mean a reductive group over $\QQ$
together with the $G(\RR)$-conjugacy class of a morphism 
$s_{0}:\SSS \rightarrow G_{\RR}$ such that $(G,X)$ satisfies the conditions (2.1.1.1), (2.1.1.2)
and (2.1.1.3) of Deligne \cite{De2}.  The Mumford-Tate group of $MT(s)$ of $s\in X$ is defined as the 
smallest $\QQ$-subgroup of $G$ such that $s:\SSS\rightarrow G_{\RR}$ factors through
$MT(s)_{\RR}$. The generic Mumford-Tate group $MT(X)$ of $X$ is defined
as the smallest $\QQ$-subgroup of $G$ containing all the $MT(s)$ for $s\in X$.
We will only consider Shimura data $(G,X)$ such  that  $G=MT(X)$; see \cite{UlYa} lemma 2.1. for the fact that
imposing this condition does not cause any loss of generality. Under this assumption the conditions
(2.1.1.4) and (2.1.1.5) of Deligne \cite{De2} are also satisfied; see the appendix to this paper.

If $K$ is a compact open subgroup of $G(\AAA_{f})$ then $ \Sh_{K}(G,X)$
is the canonical model of the Shimura variety over the reflex field $E(G,X)$
of $(G,X)$. Then 
$$
\Sh_{K}(G,X)(\CC)=G(\QQ)\backslash X\times G(\AAA_{f})/K.
$$
Let  $\Sh(G,X)$ be the projective limit of the $\Sh_{K}(G,X)$.
Then $\Sh(G,X)$ is defined over $E(G,X)$ and with our conventions (see \cite{De2} cor. 2.1.11
or \cite{Milne} thm. 5.28)
$$
\Sh(G,X)(\CC)=G(\QQ)\backslash X\times G(\AAA_{f}).
$$
We will write $[s,\ol{g}]$ for the point of $\Sh_K(G,X)(\CC)$ which is the
image of the element $(s,g)$ of $X\times G(\AAA_f)$.
We will write $[s,g]$ for a point of $\Sh(G,X)(\CC)$ which is the 
image of $(s,g)$. A point $[s,\ol{g}]$ of $\Sh_K(G,X)(\CC)$
is said to be Hodge generic if $MT(s)=G$.

\subsection{Contents of the paper}
In section two we consider a Shimura datum $(G,X)$ and  $K$  a neat compact open subgroup of 
$G(\AAA_f)$ which is a product $K=\prod_l K_l$ over primes $l$.
Let 
$\Sh_K(G,X)$ the corresponding Shimura variety over $E(G,X)$.
Let $L$ be a number field containing $E(G,X)$  and 
$x = [s,\ol{1}] \in \Sh_K(G,X)(L)\subset \Sh_K(G,X)(\CC)$.

We let $\wt{x}$ be the point $[s,1]$ of  $\Sh(G,X)$.
  We show that these choices determine a Galois representation
$$
\rho_{\wt{x}}\colon \Gal(\ol L/L) \lto K
$$
and for any prime $l$ we obtain a local representation $\rho_{\wt{x}, l}$
by projecting on the $l$-th component.
We show that in the case where the Shimura variety is 
the moduli space for principally polarised abelian varieties (with level structure),
the representations we obtain are the ones naturally associated to abelian varieties.

The point $s \colon \SSS \lto G_\RR$  of $X$ determines the 
Mumford-Tate group $M:=MT(s)$ of $s$. 
We show (prop. \ref{prop1}) quite directly that 
--at least for $L$ big enough-- 
$\rho_{\wt{x}}(\Gal(\ol L/L))$ is contained in $M(\AAA_f)$ (prop. \ref{prop1}).
Usually, for abelian varieties, this statement is derived from Deligne's result on absolute Hodge cycles.
If $H_{l}$ is the associated $l$-adic monodromy-group, it's a simple consequence
of the theory of complex multiplication (see remark \ref{remCM}) and some functoriality
properties of our definition of Galois representations 
 that $Z_{M}^0\otimes \QQ_{l}\subset H_{l}^0$. 
 We don't know a precise reference for this result, even in the case of abelian varieties,
 but this was  probably known to experts and some related statements are obtained by
 Serre\cite{Se}, Chi \cite{Ch} and Pink \cite{Pi}.

In section $3$, we show that given two points $x=[s, \ol{1}]$ and $y = [s, \ol{g}]$ lying in one Hecke 
orbit, the corresponding Galois representation are conjugate by $g$ \emph{up to multiplication}
by an element $z$ of uniformly bounded  order of $Z_{M}(\QQ)$ (see thm. \ref{teo1}). We think that $z$ might be in fact trivial but we have  not been  able to prove this.

In  section $4$, we define the notion of $(\infty, l)$-integral structure as follows.
Let $V_\QQ$ be a $\QQ$ vector space and $L$ a number field.
An $(\infty, l)$-integral structure $(V_{\ZZ},s,\rho)$ on $V_{\QQ}$ is the following data
\begin{enumerate}
\item A lattice $V_{\ZZ}$ in $V_{\QQ}$.
\item A polarised $\ZZ$-Hodge structure ($\ZZ$-PHS) $s \colon \SSS \lto \GL(V_{\RR})$.
\item An $l$-adic Galois representation $\rho \colon \Gal(\ol L/L) : \lto \GL(V_{\ZZ_{l}})$.
\end{enumerate}

Furthermore, we require the following compatibility relation
$$
\End_{\ZZ-HS}(V_{\ZZ})\otimes \ZZ_{l} \subset \End_{\Gal(\ol L/L)}(V_{\ZZ_{l}}).
$$
A $(\infty,l)$-integral structure  $(V_{\ZZ},s,\rho)$ is called Tate if  
 the representation $\rho \otimes \QQ_l$ is semi-simple and
$$
\End_{\ZZ-HS}(V_{\ZZ})\otimes \ZZ_{l} = \End_{\Gal(\ol L/L)}(V_{\ZZ_{l}})
$$
One defines a $(\infty, l)$-rational structure as above afer tensoring everything
with $\QQ$. Morphisms of $(\infty, l)$-rational structures are defined in an obvious way. 

Such a data is naturally attached to an abelian variety and more generally to a 
point $x=[s,\ol{1}]$ in $\Sh_K(G,X)(L)$ as above, \emph{and} a faithful representation 
$G\hookrightarrow \GL_{n}$. We call such a $(\infty, l)$-integral structure of Shimura type.
The generalised Mumford-Tate conjecture  asserts that for a point $x=[s,\ol{1}]$ as above,
the image of $\rho_{\wt{x},l}$ is open in $M(\QQ_l)$ where $M$ is the Mumford-Tate 
group of $s$. In the case of an $(\infty, l)$-integral structure attached to 
an abelian variety or to a point in a  Shimura variety of Hodge type, our generalised Mumford-Tate conjecture is just the usual Mumford-Tate conjecture.
   
The Isogeny conjecture asserts that any $(\infty, l)$-integral structure of Shimura type
is Tate. This conjecture is really more general than the usual Tate isogeny conjecture, proved by Faltings 
for abelian varieties.
Indeed, Faltings proved the Isogeny conjecture for $(\infty, l)$-integral structures
associated to points on moduli space of abelian varieties \emph{with respect to}
the standard symplectic representation of the symplectic group. Our Isogeny conjecture
asserts that the conclusion holds \emph{for every} representation of the symplectic group, in
particular for all representations obtained via tensor constructions with
the standard one.
We show in the fourth section that the generalised Mumford-Tate and Tate conjectures are equivalent.
The implication ``Mumford-Tate implies Tate'' is straightforward
and is valid for the classical versions of these conjectures.
This is a simple consequence of the fact that the Mumford-Tate group is reductive and that the endomorphisms
of the Hodge structures are precisely those commuting with the action of the
Mumford-Tate group.
The other implication really uses the fact that the conclusion of the Tate conjecture
is true for every representation.

The generalised Shafarevich conjecture is the statement that an isogeny class of a 
$(\infty, l)$-integral structure contains only finitely many isomorphism classes.
We show in the fith section that the generalised Shafarevich conjecture implies the Isogeny conjecture (see section 5.1).
We follow ad-hoc arguments of Tate from \cite{Tate}.

The implication that the generalised Mumford-Tate conjecture implies the generalised Shafarevich conjecture ( theorem \ref{teo5.5}) is
much harder and totally new. It uses a non-trivial result on conjugacy classes of open subgroups
in $p$-adic groups (proposition \ref{prop5.8}) which, we believe, is of independent interest.

\subsection{Acknowledgements.}

The second author is very grateful to the Universit\'e de Paris-Sud for several invitations
and to the Leverhulme Trust for financial support. 

\section{Galois representations attached to points on Shimura varieties.}

In this section we explain how to attach a Galois representation to
a rational point on a Shimura variety.
Some simple consequences of the definition are given. We prove
that the neutral component of the $l$-adic monodromy group is contained
in the Mumford-Tate group (proposition \ref{prop1}) and that the  connected
centre of the Mumford-Tate group
is included in the connected centre of $l$-adic monodromy group (corollary \ref{centreMT2})
and coincides with it for Shimura varieties of Hodge type.

\subsection{Mumford-Tate groups and Shimura varieties.}

Let $(G,X)$ be a Shimura datum and $K$ a compact open subgroup
of $G(\AAA_f)$.

We let
$$
\pi \colon \Sh(G,X)\lto \Sh_K(G,X)
$$
be the projection.  Then $\pi$ is defined over the reflex field $E:=E(G,X)$ of $(G,X)$.

 Let $L$ be a finite extension
of $E$ and  $x$ be a point in $\Sh_K(G,X)(L)$.

\begin{lem}\label{lemme2.1}
Write $x = [s,\ol{g}]$ with $s$ in $X$ and $g$ in $G(\AAA_f)$.

(a) Assume that $K$ is neat.
The fibre $\pi^{-1}(x)$   has a simply transitive right
action by $K$.

(b) Assume that $x$ is Hodge generic. The group $Z_{G}(\QQ)\cap K$
is finite and the fibre $\pi^{-1}(x)$ has a simply transitive right action
by 
$$K/Z_{G}(\QQ)\cap K.
$$ 

\end{lem}
\begin{proof}

We first check that the fibre $\pi^{-1}(x)$ is $[s, gK]$.

Let $[s', g']$ be a point in $\pi^{-1}([s,\ol{g}])$.
There exists an element $q$ of $G(\QQ)$ and $k$ in $K$ such that
$$
s' = q s, \ g' = qgk
$$
We have $[qs,qgk]=[s,gk]$ in $\Sh(G,X)$.
The group $K$  acts by right multiplication and
this action is transitive.
Suppose that $[s,g]= [s, gk]$ with $k$ in $K$. 
There exists a $q$ in $G(\QQ)$ such that $gk = qg$ and 
$q.s=s$. Therefore 
\begin{equation}\label{eq2.1}
q\in G(\QQ)\cap Z_{G_{\RR}}(s(\SSS))(\RR)\cap gKg^{-1}.
\end{equation}
 
We make use of the following lemma.
\begin{lem}\label{lemme2.2}
Let $M$ be the Mumford-Tate group of $s$ and
$C:=Z_{G}(M)$.  Then 
$$
G(\QQ)\cap Z_{G_{\RR}}(s(\SSS))(\RR)=C(\QQ).
$$
\end{lem}
\begin{proof}
Fix a faithful representation $G\rightarrow \GL(V_{\QQ})$ on a $\QQ$-vector space
$V_{\QQ}$. Then $s:\SSS\rightarrow G_{\RR}\rightarrow \GL(V_{\RR})$ defines a  Hodge structure on
$V_{\QQ}$ see \cite{De2} 1.1.11--14. Let $\End_{HS}(V_{\QQ})$ denotes the endomorphisms
of $\QQ$-Hodge structures on $V_{\QQ}$. Then
$$
\End_{HS}(V_{\QQ})=\End_{s(\SSS)}(V_{\QQ})=\End_{M}(V_{\QQ})
$$
where, by slight abuse of notation, by $\End_{s(\SSS)}(V_{\QQ})$ we mean endomorphisms $f$ of $V_{\QQ}$
such that $f_{\RR}$ commutes with the action of $s(\SSS)$.
Therefore
$$
\GL(V_{\QQ})\cap Z_{\GL(V_{\RR})}(s(\SSS))(\RR)=Z_{\GL(V_{\QQ})}(M)(\QQ).
$$
The lemma is then obtained by taking the intersection
with $G(\QQ)$ in the last equation.

\end{proof}

We can now finish the proof of the lemma \ref{lemme2.1}.
Let $[s,g]$ be an element of  $\pi^{-1}(x)$ and $k$ an element of $K$  such that
$[s, gk]=[s,g]$. By the previous lemma and the equation (\ref{eq2.1})
there exists a   
$$
q\in C(\QQ)\cap gKg^{-1}
$$
 such that $q.s=s$ and $gk = qg$.  

Assume first that $x$ is Hodge generic.   
By the previous lemma $q\in Z_{G}(\QQ)\cap K$. As $Z_{G}(\QQ)\cap K$
acts trivially on $\pi^{-1}(x)$, we get a simply transitive  right action
of $K/Z_{G}(\QQ)\cap K$ on $\pi^{-1}(x)$. Any compact open 
subgroup of $G(\AAA_{f})$  contains a neat compact open subgroup of $G(\AAA_{f})$
(see \cite{Pi} 0.6).
To prove that $K\cap Z_{G}(\QQ)$ is finite, we just need to prove that 
$K\cap Z_{G}(\QQ)$ is trivial when $K$ is neat. 

Assume now that $K$ is neat.
Write $C = Z_{G} C'$ its decomposition as an almost direct product.
Let $K_C = C(\AAA_f)\cap K$, $K_Z = Z_{G}(\AAA_f)\cap K$ and $K_{C'}=C'(\AAA_f)\cap K$.
Let $q$ be an element of $K_C\cap C(\QQ)$. There is an integer $n$, an element $z$ of $Z_{G}(\QQ)\cap K_Z$
and an element $c'$ of $K_{C'}\cap C'(\QQ)$ such that
$q^n = z c'$.

 As $Z_{G}(\QQ)$ is discrete in $Z_{G}(\AAA_f)$ and $K_Z$ is neat, we have that $z=1$.
Similarly, using the fact that $C'(\RR)$ is compact, we get $c'=1$.
Therefore $q^n = 1$. As $K$ is neat, it follows that $q=1$. 
We have proved that $C(\QQ)\cap K = \{ 1\}$.

When $K$ is neat, $gKg^{-1}$ is neat and the previous proof gives
$C(\QQ)\cap gKg^{-1}=\{1\}$.
This shows that the fibre $\pi^{-1}(x)$  has a simply transitive  right
action by $K$. 

\end{proof}

Note that we have proved the following lemma that will be used later on.

\begin{lem} \label{lem2.2}
For any neat compact open subgroup $K$ of $G(\AAA_f)$,
$C(\QQ)\cap K = \{1\}$.
\end{lem}

The fibre $\pi^{-1}(x)$  also has a left action by $\Gal(\ol L/L)$ which commutes with the 
action of $K$. 

We now make use of the following elementary lemma.

\begin{lem} \label{superlemma}
Let $S$ be a set. Suppose that a group $H$ acts simply transitively on $S$ on the right.
Suppose that a group $G$ acts on $S$ on the left and that the action of $G$ commutes with that of $H$.

Fix a point $s_0 \in S$.
There is a unique  homomorphism 
$\phi_0 \colon G \lto H$ such that for $g\in G$, 
$$
g\cdot s_0 = s_0 \cdot \phi_0(g)
$$
Furthermore, let $s$ be a point of $S$ and $h\in H$ such that $s = s_0 \cdot h$.
Let $\phi$ be  the homomorphism  $\phi \colon G \lto H$ such that for $g\in G$, 
$g\cdot s = s \cdot \phi(g)$
then $\phi=h^{-1} \phi_0 h$.
\end{lem}





\subsection{Adelic and $l$-adic  Galois representations}

We can now define Galois representations 
associated to rational points on Shimura varieties.
This will be the main object of study in this paper.

Assume first that $K$ is neat
and let $L$ be a number field containing 
the reflex field $E:=E(G,X)$ of the Shimura datum $(G,X)$.
Let
$$
x = [s,1]\in \Sh_{K}(G,X)(L)
$$   
with $s$ and element of $X$ and let $\widetilde{x} = [s,k]$ 
be a point in $\pi^{-1}(x)$ 
which is the image of $(s,k)$ in $\Sh(G,X)$. 
By applying the above lemma to our situation, we see 
 that there is a morphism
\begin{equation}
\rho_{\wt{x}} \colon \Gal(\ol L/L)\lto K
\end{equation}
describing the Galois action on $\pi^{-1}(x)$.
Let $\sigma$ be an element of $\Gal(\ol L/L)$, then
$$
\sigma(\widetilde{x}) = [s,k \rho_{\wt{x}}(\sigma)].
$$

If $x$ is Hodge generic then the same method gives a morphism
$$
\rho_{\wt{x}} \colon \Gal(\ol L/L)\lto K/Z_{G}(\QQ)\cap K
$$
such that 
for all $\sigma\in\Gal(\ol L/L)$
$$
\sigma(\widetilde{x}) = [s,k \rho_{\wt{x}}(\sigma)].
$$
This last formula makes sense as $Z_{G}(\QQ)\cap K$
acts trivially on $\Sh(G,X)$.

Fix a prime number $l$. Let $p_{l}: G(\AAA_f)\rightarrow G(\QQ_{l})$ be the  
projection. Let  $K^l$ be the kernel of the restriction of $p_{l}$ to $K$
and $K_{l}:=p_{l}(K)$.

Let 
$$
\Sh_{K^l}(G,X) := G(\QQ)\backslash X\times G(\AAA_f)/K^l
$$
and
$$
\pi_l \colon \Sh_{K^l}(G,X)\lto \Sh_{K}(G,X)
$$
be the natural projection.
We have the following commutative diagram:
 \[
 \xymatrix{
 \Sh(G,X) \ar[r]^{\pi^l} \ar[rd]^{\pi} &\Sh_{K^l}(G,X)\ar[d]^{\pi_l}\\
  & \Sh_K(G,X)
 }
 \]
where the horizontal map is the quotient by $K^l$. 

 Finally let  $p^l:G(\AAA_{f})\rightarrow G(\AAA_{f}^l)$ be the projection.
If $g\in G(\AAA_{f})$ we will write $g_{l}=p_{l}(g)$ and $g^l=p^l(g)$
and $[s,g_{l},\ol{g^l}]
\in \Sh_{K^{l}}(G,X)(\CC)$ for the 
image of $(s,g)\in X\times G(\AAA_{f})$. 

We recall that $\widetilde{x} = [s,k]$ 
is a point in $\pi^{-1}(x)$. Let  
$$
\widetilde{y}=\pi^l(\widetilde{x}) =[s,k_{l},\overline{k^l}]\in \pi_{l}^{-1}(x).
$$

Assume that $K$ is neat.
We define the $l$-adic Galois representation
$\rho_{\wt{x},l}:\Gal(\oL/L)\rightarrow K_{l}$ by the formula
$$
\rho_{\wt{x},l} = p_l \circ \rho_{\wt{x}}: \Gal(\ol L/L)\rightarrow K_{l}.
$$

\begin{lem}
The $l$-adic Galois representation $\rho_{\wt{x},l}$ describes the Galois action on $\pi_{l}^{-1}(x)$:
for all $\sigma\in\Gal(\ol L/L)$
$$
\sigma(\widetilde{y}) = [s,k_{l} \rho_{\wt{x},l}(\sigma), \overline{k^lp^l(\rho_{\wt{x}} (\sigma))}].
$$
Remark that $K^l$ is of finite index in $\tilde{K}^l:=p^l(K)$ 
therefore $\overline{k^lp^l(\rho_{\wt{x}} (\sigma))}$ varies in a finite set.
When $K=K_{l}K^l$, we have $\tilde{K}^l=K^l$ and
$\overline{k^lp^l(\rho_{\wt{x}} (\sigma))}=\ol{1}$. Therefore 
we simply get
$$
\sigma(\widetilde{y}) = [s,k_{l} \rho_{\wt{x},l}(\sigma),\ol{1}].
$$
Note that for all $l$ big enough, we have $K=K_{l}K^l$.
\end{lem}

\begin{proof}
The map $\pi^l:\Sh(G,X) \lto \Sh_{K_l}(G,X)$ is defined over $E$
and sends $\wt{x}=[s,k] \in \Sh(G,X)$ to $\wt{y}=[s,k_{l},\ol{k^l}]\in\Sh_{K^{l}}(G,X)$).
For any $\sigma\in \Gal(\ol L/L)$, 
$\sigma(\wt{x})=[s,k \rho_{\wt{x}}(\sigma)]$.

 As  $\pi^l(\sigma(\wt{x}))= \sigma(\pi^l(\wt{x}))= \sigma(\wt{y})$
 we find the formula of the lemma.  
\end{proof}

\begin{defini}
When $x=[s,\ol{1}]\in \Sh_{K}(G,X)(L)$ and $\wt{x}=[s,1]$ we say that
the Zariski closure $H_{l}$ of $\rho_{\wt{x},l}(\Gal(\oL/L))$ in $G_{\QQ_{l}}$
is  the $l$-adic monodromy group of $\rho_{\wt{x}}$. 
\end{defini}

\begin{rem}\label{remCM}{\rm
The relation with the usual theory of complex multiplication is the following.
Let $(T,\{s\})$ be a CM-Shimura datum with reflex field $E=E(T,\{s\})$. Let $K_{T}$ be
a neat open compact subgroup of $T(\AAA_{f})$,
$x=[s,\overline{1}]\in \Sh_{K_{T}}(T,\{s\})(L)$
for a finite extension $L$ of $E$ and $\wt{x}=[x,1]\in \Sh(T,\{s\})$. 
The Galois action on $\Sh(T,\{s\})$ is given by a reciprocity morphism
$$
r:\Gal(\ol{L}/L)^{ab}\rightarrow \ol{T(\QQ)}\backslash T(\AAA_{f})
$$
by the formula $\sigma([s,1])=[s,r(\sigma)]$.
As 
$$
\pi:\Sh(T,\{s\})\rightarrow \Sh_{K_{T}}(T,\{s\})
$$
is defined over $L$, we get 
$$\pi(\sigma(\wt{x}))=\sigma(x)=x=[s,\ol{1}] = [s, \ol{r(\sigma)}].$$
Therefore $r(\sigma)\in T(\QQ) K_{T}$. As $K_{T}$ is neat $K_{T}\cap T(\QQ)$
is trivial and we can write $r(\sigma)=t(\sigma)k(\sigma)$ with $t(\sigma)\in T(\QQ)$
and $k(\sigma)\in  K_{T}$. Then 
$$
\sigma([s,1])= [s,k(\sigma)]
$$
and we see that the action of $\Gal(\oL/L)$ on $\pi^{-1}(x)$ is given
by a morphism $\rho_{\wt{x}}:\Gal(\oL/L)\rightarrow K_{T}$, 
$\sigma\mapsto \rho_{\wt{x}}(\sigma)=k(\sigma)$.

As  a consequence of this discussion and classical results on complex multiplication
due to Shimura and Taniyama,
if $M\subset T$ is the Mumford-Tate group of $s$ then the $l$-adic monodromy
group of $\rho_{\wt{x}}$ is $M_{\QQ_{l}}$.}
\end{rem}

We now explain why our construction of Galois representations
attached to points on Shimura varieties, in the case of $\cA_{g,n}$
gives the Galois representations on the Tate module of abelian
varieties.

Suppose that $\Sh_K(G,X)$ is $\cA_{g,n}$ ($n \geq 3$), the moduli space of principally polarised abelian varieties with level $n$ structure. 
Therefore $G=\GSp_{2g}$, $K=\{\alpha\in \GSp_{2g}(\widehat{\ZZ})\vert \alpha \equiv \mbox{Id } [n]\}$. 
Let $x$ be a point in $\cA_{g,n}(L)$ where $L$ is a number field.
The point $x=[s, \ol{1}]$ corresponds to the datum $(A, \cL, \phi_n)$
where $A$ is a $g$-dimensional abelian variety over $L$, $\cL$ is
a principal polarisation and $\phi_n$ is a symplectic isomorphism  
$$
\phi_{n}\colon A_n(\ol L)\lto (\ZZ/n\ZZ)^{2g}
$$
where $A_{n}$ denotes the kernel of the multiplication by $n$ map on $A$.

Choose the point $\wt{x} = [s,1]$ in $\Sh(G,X)$.
This point corresponds to $(A,\cL, \wt{\phi})$  
with $A$ and $\cL$ the same as above and
$$
\wt{\phi} \colon T(A) \lto \widehat{\ZZ}^{2g}
$$
is a symplectic isomorphism inducing $\phi_n$ on $A_n$ by reduction.
Here $T(A) = \prod_{l} T_lA$ with $T_lA$ being the $l$-adic Tate module
(or to put it another way, $T(A) = \varprojlim_{n \in \NN} A_n(\ol L)$).
The fibre $\pi^{-1}([s,\ol{1}])$ is the $K$-orbit of $\wt{\phi}$ where
for $k\in K$, $\wt{\phi}\cdot k$ is the composite
$$
T(A) \overset{\wt{\phi}}{\lto} \widehat{\ZZ}^{2g} \overset{k}{\lto} \widehat{\ZZ}^{2g}
$$
and the points
$[\wt{s},k]$ correspond  to $(A,\cL, \wt{\phi}\cdot k)$.
Notice that giving the isomorphim $\wt{\phi}$ is equivalent to
choosing a symplectic basis for $T(A)$. Furthermore, the action
of $\sigma$ in $\Gal(\ol L/L)$ on $T(A)$ is determined by its
effect on a basis.
On the other hand, the set of symplectic bases of $T(A)$
has a transitive action by $\GSp_{2g}(\widehat{\ZZ})$. 
The Galois and $\GSp_{2g}(\widehat{\ZZ})$ actions commute.
The condition that the
bases we consider reduce mod $n$ to a fixed given base, implies 
that the restriction of this  action to $K$ has no fixed points.

We see, by the uniqueness of the morphism in \ref{superlemma}, 
that the Galois representation $\rho_{\wt{x}}$
is the representation $\rho_A$ on $T(A)$ obtained from the
action of $\Gal(\ol L/L)$ on each $A_n(\ol L)$.
The same  holds for $l$-adic representations 
$\rho_{\wt{x}, l}$ and $\rho_{A,l}$ at least for $l$ prime to $n$.

\begin{prop} \label{prop1}
Assume that $K$ is neat.
Let $x=[s,\ol{1}]$ be a point of $\Sh_K(G,X)(L)$ and let $\rho_{\wt{x}}$ be the Galois representation
attached to $\wt{x}=[s,1]$. Let $M$ be the Mumford-Tate group of $s$. 
Let $L'$ be the subfield of $\overline{\QQ}$ generated by $L$ and the reflex field $E(M,X_{M})$.
Then 
$$
\rho_{\wt{x}}(\Gal(\ol L/L'))\subset M(\AAA_f)\cap K
$$
and by compatibility,
$$
\rho_{\wt{x},l}(\Gal(\ol L/L'))\subset M(\QQ_l)\cap K_l.
$$
\end{prop}
\begin{proof}
Let $X_M$ be the $M(\RR)$-conjugacy class of $s$, then $(M,X_M)$ is a sub-Shimura datum
of $(G,X)$. Then $\Sh_{K\cap M(\AAA_f)}(M,X_M)$ is a sub-Shimura variety containing $x$.
The projective limit $\Sh(M,X_M)$ is a subvariety of $\Sh(G,X)$.

We have the following commutative diagram
$$
\begin{array}{ccccl}
& & \wt{x}  \in \Sh(M,X_M) & \hookrightarrow
& \Sh(G,X) \\
& &  \downarrow & & \downarrow  \\
& & x\in \Sh_{K\cap M(\AAA_f)}(M,X_M) & \hookrightarrow & \Sh_K(G,X)
\end{array}
$$

 As $M$ is the generic Mumford-Tate group on $X_M$, from lemma 5.1 of \cite{KliYa}, it follows that $Z_{G}^0\subset Z_M^0$. Then $Z_M^0$ is isogeneous to $Z_{G}^0 \times T$ where $T$ is a subtorus of $Z_M^0$ satisfying
the condition that $T(\RR)$ is compact.
As $Z_{G}^0(\QQ)$ and $T(\QQ)$ are discrete in $Z_{G}^0(\AAA_f)$
 and $T(\AAA_f)$ respectively, it follows that
the same holds for $Z_M^0(\QQ)$ inside $Z_M^0(\AAA_f)$ (c.f. appendix).
Then there exists a representation $\rho_{M,\wt{x}} \colon \Gal(\ol L / L')\lto M(\AAA_f)\cap K$.
As in the above diagram, all arrows are defined over $L'$, $\rho_{\wt{x}}$ factorises through $\rho_{M,\wt{x}}$.
\end{proof}

The definition of the Galois representations is functorial in the following sense.
Let $\theta:(H,X_{H})\rightarrow (G,X)$ be a morphism of Shimura data. 
Let $K_{H}$ and $K$ be neat open compact subgroups of $H(\AAA_{f})$
and $G(\AAA_{f})$ and assume that $\theta(K_{H})\subset K$.
We then get  associated morphisms of Shimura varieties:
$$
f:\Sh_{K_{H}}(H,X_{H})\rightarrow \Sh_{K}(G,X)
$$
and 
$$
\wt{f}:\Sh(H,X_{H})\rightarrow \Sh(G,X).
$$
Let $L$ be a number field containing the reflex fields of $(G,X)$
and  $(H,X_{H})$, then $f$ and $\wt{f}$ are defined over $L$.
Let
$$
x=[s,\ol{1}]\in \Sh_{K_{H}}(H,X_{H})(L),
$$
$$
y=f(x)=[t,\ol{1}]\in \Sh_{K}(G,X)(L),
$$ 
$\wt{x}=[s,1]\in \Sh(H,X_{H})(\oL)$
and $\wt{y}=\wt{f}(\wt{x})\in\Sh(G,X)(\oL)$. Then 
using that the commutative diagram
$$
\begin{array}{ccccl}
& &    \Sh(H,X_H) & \rightarrow
& \Sh(G,X) \\
& &  \downarrow & & \downarrow  \\
& &  \Sh_{K_{H}}(H,X_H) & \rightarrow & \Sh_K(G,X)
\end{array}
$$
is defined over $L$, we get
\begin{equation}
\rho_{\wt{y}}=\theta \rho_{\wt{x}}.
\end{equation}
Let $K^l$ (resp. $K_{H}^l$) be the kernel of the projections of $K$
(resp. $K_{H}$) on $G(\QQ_{l})$ (resp $H(\QQ_{l})$)
 Let  $K_{l}$ and $K_{H,l}$ be the image of these projections.
We have an induced morphism $\theta_{l}: K_{H,l}\rightarrow K_{l}$
and the relation
\begin{equation}
\rho_{\wt{y},l}=\theta_{l} \rho_{\wt{x},l}
\end{equation}

\begin{cor}
Let $H_{l}$ (resp $H'_{l}$) be the $l$-adic monodromy groups 
of $\rho_{\wt{x}}$ (resp. $\rho_{\wt{y}})$. Then
$H'_{l}=\theta_{l}(H_{l})$.
\end{cor}

An interesting case of the corollary is the following.
Let $$
x=[s,\ol{1}]\in \Sh_{K}(G,X)(L),$$
$\wt{x}=[s,1]$ and 
$\rho_{\wt{x}}$ the associated Galois representation.

We have a decomposition as an almost direct product $G=Z_{G}^0G^\der$. 
Let $G^{ab}$ be the torus $G/G^\der$. 
To $(G,X)$ one associates two Shimura data $(G^{ab},\{c\})$ and
$(G^\ad, X^{\ad})$. The reflex field $E(G,X)$ is the composite of
$E(G^{ab},\{c\})$ and $E(H^\ad,X^{\ad})$. There are morphisms
of Shimura data
$$
\theta^\ab\colon (G,X)\longrightarrow (G^{ab},\{c\}) \mbox{ and }
\theta^\ad\colon (G,X)\longrightarrow (G^\ad, X^{\ad}).
$$
Notice that refering to $(G^{ab},\{ c \})$ as `Shimura datum' is a slight abuse of 
terminology. Indeed, it can (and does) happen that $G^{ab}$ is $\GG_m$,
in which case $c \colon \SSS \lto \GG_{m\RR}$ is a power of the norm morphism.
However, for a compact open subgroup $K_{G^{ab}}$ of $G^{ab}(\AAA_f)$, the set
$\Sh_{K_{G^{ab}} } (G^{ab},\{ c \}):= G^{ab}(\QQ)\backslash G^{ab} \times \{ c\} / K_{G^{ab}}$
is well-defined and has a well-defined Galois action via 
a reciprocity morphism, as described by Deligne.
In particular, given a point $x = [s,1]$ in
$\Sh_{K_{G^{ab}}}(G^{ab},\{ c \})(L)$, the Galois representation $\rho_{\wt{x},l}$ (where $\wt{x} = [s,1]$)
defined as before has the property that the image of $\rho_{\wt{x},l}$ in $G^{ab}(\QQ_l)$ is 
open, just like in the case where $(G^{ab}, \{ c \})$ is a special Shimura datum.

Fix some open compact subgroups $K_{G^{ab}}\subset G^{ab}(\AAA_{f})$
and $K^{ad}\subset H^{ad}(\AAA_{f})$ containing 
$\theta^{ab}(K)$ and $\theta^{ad}(K)$ respectively.

\begin{cor}\label{centreMT}
Let $x^{ad}$ (resp. $x^{ab}$) the images of $x$ in 
$\Sh_{K^{ad}}( G^\ad, X^{\ad})$ (resp. $\Sh_{K_{G^{ab}}}(G^{ab},\{c\})$).
In the same way, we define $\wt{x}^{ad}$ and $\wt{x}^{ab}$.
Let $H_{l}$, $H^{ad}_{l}$ and $H^{ab}_{l}$ be the $l$-adic monodromy groups 
of $\rho_{\wt{x}}$, $\rho_{\wt{x}^{ad}}$,$\rho_{\wt{x}^{ab}}$
respectively.
Then
$$
H^{ad}_{l}=\theta^{ad}(H_{l}) \mbox{ and } H^{ab}_{l}=\theta^{ab}(H_{l}).
$$ 
Using the remark \ref{remCM}
we find that 
$\theta^\ab (H_{l})=G^{ab}_{\QQ_{l}}$ and  that $Z_G^0\otimes {\QQ_{l}}\subset H_{l}$.
\end{cor}

\begin{cor}\label{centreMT2}
Assume that $K$ is neat.
Let $x=[s,\ol{1}]$ be a point of $\Sh_K(G,X)(L)$ and let $\rho_{\wt{x}}$ be the Galois representation
attached to $\wt{x}=[s,1]$. Let $M$ be the Mumford-Tate group of $s$. Let $H_{l}$ be 
the associated $l$-adic monodromy group then
\begin{equation}
Z_{M}^0\otimes \QQ_{l}\subset Z_{H_{l}}^0.
\end{equation}\label{MT1}
If $(G,X)$ is a sub-Shimura datum of 
$(\GSp_{2g},\HH_{g}^{\pm})$ (i.e if  $\Sh_K(G,X)$ is of Hodge type)
then
\begin{equation}
Z_{M}^0\otimes \QQ_{l}= Z_{H_{l}}^0.
\end{equation}
\end{cor}
\begin{proof}
Using proposition \ref{prop1} and it's proof we may assume that $(G,X)=(M, X_{M})$
and the inclusion $Z_{M}^0\otimes \QQ_{l}\subset Z_{H_{l}}^0$ is a consequence
of the corollary \ref{centreMT}. If moreover $(G,X)\subset (\GSp_{2g},\HH_{g}^{\pm})$,
then $M$ and $H_{l}$ are the Mumford-Tate group and the $l$-adic monodromy group
of an abelian variety $A$ over a number field $L$. By Faltings result on the Tate conjecture \cite{Fa}
$$
\End_{L}(A)\otimes \ZZ_{l}=\End_{\Gal(\ol L/L)}(T_{l}(A))=\End_{H_{l}}(T_{l}(A))=\End_{M\otimes \QQ_{l}}(V_{l}(A)).
$$
Therefore $Z_{H_{l}}^0\subset Z_{M}^0\otimes \QQ_{l}$.
\end{proof}

\section{Galois action on a  Hecke orbit.}

Let $x = [s, \ol{1}]$ be a point of $\Sh_K(G,X)(L)$ (as before, in particular recall that $K$ is neat) and let $M$ be the
Mumford-Tate group of $s$. We assume that $L$ contains $E(M,X_M)$ so that
the conclusion of \ref{prop1} holds.
In this section we study how $\Gal(\ol L/L)$ acts on a $L$-rational point in the
generalised Hecke orbit of $x$. 

Let $y \in \Sh_K(G,X)(L)$ in the generalised Hecke orbit of $x$.
Let $g\in G(\QQ)$ such that $y:=[s,\ol{g}]\in  \Sh_K(G,X)(L)$ and
let  $\wt{y} = [s,g] \in \pi^{-1}(y)$. 

As in section 2.2, by applying lemma \ref{superlemma} to our situation, we obtain
 the representation
 $\rho_{\wt{y}} \colon \Gal(\ol L/L) \lto K$
having the property that for $\sigma\in \Gal(\ol L/L)$ and $k\in K$, we have
$$
\sigma([s,gk]) = [s, g k\rho_{\wt{y}}(\sigma)]
$$

\begin{teo} \label{teo1}
Let $x=[s,\ol{1}]$ be a point of $ \Sh_K(G,X)(L)$.
Then, for any $\sigma$ in $\Gal(\ol L/L)$, there 
exists an element $c_\sigma$ of finite order in $Z_{M}^0(\QQ)$ such that
$$
\rho_{\wt{y}}(\sigma) = c_{\sigma} g^{-1} \rho_{\wt{x}} g
$$
where $\wt{x} = [s,1]$.
Furthermore, the set $\{ c_\sigma : \sigma \in \Gal(\ol L/L) \}$ is a finite group
of uniformly bounded order (i.e not depending on $\wt{x}$, $\wt{y}$).
\end{teo}

\begin{proof}
The action of $\rho_{\wt{y}}$ on $[s,g]$ is given by 
$$
\sigma([s,g]) = [s, g \rho_{\wt{y}}(\sigma)]
$$
On the other hand
$$
[s,g] = [s,1] \cdot g
$$
therefore, using the fact that Galois commutes with Hecke, we get
$$
\sigma([s,g]) = (\sigma[s,1]) \cdot g = [s, \rho_{\wt{x}}(\sigma) g]
$$
Hence we have 
$$
[s, g \rho_{\wt{y}}(\sigma)] = [s, \rho_{\wt{x}}(\sigma) g]
$$

We recall that  we defined in the last section $C:=Z_{G}(M)$. By Lemma \ref{lemme2.2}
there exists $c_\sigma \in C(\QQ)$ such that 
$$
c_\sigma \rho_{\wt{x}}(\sigma) g = g \rho_{\wt{y}}(\sigma)
$$

\begin{lem}
The set $\Theta:=\{c_\sigma, \sigma \in \Gal(\ol L/L)\}$ is a group.
\end{lem}
\begin{proof}
We will see that the $c_\sigma$ satisfy the following relation
$$
c_{\sigma\sigma'} = c_\sigma c_{\sigma'}
$$
for all $\sigma, \sigma' \in \Gal(\ol L/L)$.
By definition $c_\sigma = g \rho_{\wt{y}}(\sigma) g^{-1} \rho_{\wt{x}}(\sigma)^{-1}$.
An easy calculation shows that 
$$
c_{\sigma\sigma'}= c_{\sigma} \rho_{\wt{x}}(\sigma) c_{\sigma'} \rho_{\wt{x}}(\sigma)^{-1}
$$
By the proposition \ref{prop1}, the image of $\rho_{\wt{x}}$ is contained in $M(\AAA_f)$.
The fact that $c_\sigma$ centralizes $M(\AAA_f)$ implies that $c_{\sigma \sigma'} = c_{\sigma} c_{\sigma'}$.

It is obvious that $c_1 = 1$, therefore the lemma is proved.
\end{proof}
 
 \begin{lem}
We have the following properties:
\begin{enumerate} 
\item 
Let $K_{M}=M(\AAA_{f})\cap K$ and $K'_{M}=M(\AAA_{f})\cap gKg^{-1}$.
Then $c_{\sigma}\in K'_{M}\cdot K_{M}$.
\item  $c_{\sigma}\in Z_{M}(\QQ)$.
As a consequence $\Theta$ is an abelian group.
\end{enumerate}
 \end{lem}
 \begin{proof}
 As $\wt{y}=[s,g]=[g^{-1}s,1]$ and as the Mumford-Tate group of $g^{-1}.s$
 is $g^{-1}Mg$,  proposition \ref{prop1} implies that 
 $$
 \rho_{\wt{y}} (\Gal(\oL/L))\subset g^{-1}M(\AAA_{f}) g\cap K.
 $$ 
As $c_\sigma = g \rho_{\wt{y}}(\sigma) g^{-1} \rho_{\wt{x}}(\sigma)^{-1}$, we see that 
$c_{\sigma}\in K'_{M}\cdot K_{M}\subset M(\AAA_{f})$. Therefore
$c_{\sigma}\in Z_{G}(M)(\QQ)\cap M(\AAA_{f})=Z_{M}(\QQ)$.
 \end{proof}

We have a finite decomposition
$$
K'_{M} \cdot K_M = \coprod_{i\in I} \alpha_i K_M
$$
for some $\alpha_i \in  K'_M $.

\begin{lem}
For each $i$, $\alpha_i K \cap C(\QQ)$ (and in particular  $\alpha_i K_M \cap C(\QQ)$) has at most one element.
\end{lem}
\begin{proof}
Let $c = \alpha_i k$ and  $c' = \alpha_i k'$ be two elements of $C(\QQ)$.
Then 
$$
{c'}^{-1} c = \alpha_i^{-1}{k'}^{-1} k \alpha_i \in C(\QQ)\cap \alpha_i^{-1} K \alpha_i
$$
Then lemma \ref{lem2.2} implies that $c = c'$.
\end{proof}


%

We now know that $\Theta$ is a finite subgroup of $Z_{M}(\QQ)$. 
As $K$ is neat, $K_{M}$ is neat and   $Z_{G}(\QQ)\cap K_{M}=\{1\}$. Therefore
if $c_{\sigma}\neq c_{\sigma'}$ then 
$$c_{\sigma}K_{M}\cap c_{\sigma'}K_{M}=\emptyset.$$
Therefore $\Theta K_{M}$ is a compact open subgroup of $M(\AAA_{f})$
and 
$$[K_{M}:\Theta K_{M}]=\vert \Theta\vert.$$

Choose a faithful representation $G \hookrightarrow \GL_n$. Then $\Theta$ is a finite
subgroup of $\GL_n(\QQ)$ and therefore its order is bounded in terms of $n$ only
(see thm 3 of \cite{Friedland}).
\end{proof}

\section{Tate structures.}
\subsection{Definitions.}
In this section we define the notions of $(\infty, l)$-integral structures and Tate structures. 
These generalise the data naturally attached to abelian varieties over number fields.
A point on a Shimura variety and a faithful representation of the corresponding group naturally give rise to such a $(\infty, l)$-integral structure.
We then state the Isogeny and generalised Mumford-Tate conjectures and we show that they are equivalent.
Recall the following definition given in the introduction.

\begin{defi} [$(\infty,l)$-integral structure]
Let $L$ be a number field, $V_{\QQ}$ a finite dimensional $\QQ$-vector space
and $l$ a prime number.
An $(\infty, l)$-integral structure $(V_{\ZZ},s,\rho)$ on $V_{\QQ}$ is the following data:
\begin{enumerate}
\item A lattice $V_{\ZZ}$ in $V_{\QQ}$.
\item A polarised $\ZZ$-Hodge structure ($\ZZ$-PHS) $s \colon \SSS \lto \GL(V_{\RR})$.
\item An $l$-adic Galois representation $\rho \colon \Gal(\ol L/L)= \Gal(\ol\QQ/L)\lto \GL(V_{\ZZ_{l}})$.
\end{enumerate}

Furthermore, we require the following compatibility condition
$$
\End_{\ZZ-HS}(V_{\ZZ})\otimes \ZZ_{l} \subset \End_{\Gal(\ol L/L)}(V_{\ZZ_{l}})
$$
An $(\infty,l)$-integral structure  $(V_{\ZZ},s,\rho)$ is called Tate if  
 the representation $\rho \otimes \QQ_l$ is semi-simple and

$$
\End_{\ZZ-HS}(V_{\ZZ})\otimes \ZZ_{l} = \End_{\Gal(\ol L/L)}(V_{\ZZ_{l}})
$$

We define the notion of l $(\infty,l)$-rational structure on a finite dimentional $\QQ$-vector space
$V_{\QQ}$ in a completely analogous manner.
\end{defi}

Let $(V_{\ZZ},s,\rho)$ be an $(\infty,l)$-integral structure. We let $G_l$ be the Zariski closure 
of the image of $\rho\otimes \QQ_{l}$ in $\GL(V_{\QQ_l})$ and let $G_l^0$ be the neutral component of $G_l$.
Then $\rho$ induces a representation
$$
\rho\colon \Gal(\ol L/L)\lto G_l \lto G_l/G_l^0
$$
As the group $G_l/G_l^0$ is finite , after replacing $L$ by a finite extension, we can assume that $G_l$ is connected.
We will always  make this assumption in what follows. 

The Mumford-Tate group $M$ of a $(V_{\ZZ},s,\rho)$ is defined as the smallest subgroup
$H$ of $\GL(V_{\QQ})$ such that $s$ factors through $H_{\RR}$.

Note that if $A$ is an abelian variety over $L$ and $V_{\ZZ} = H_1(A,\ZZ)$,
then we obtain a $(\infty, l)$-integral structure as follows:
$s$ is the Hodge structure naturally attached to $A$ and
$\rho \colon \Gal(\ol L/L) \lto \GL(V_{\ZZ_l})$ is the Galois representation on the Tate module
$T_lA = V_{\ZZ_l}$ attached to $A$. Then
$(V_{\ZZ}, s, \rho)$ is an $(\infty, l)$-integral structure 
which is Tate by the theorems of Faltings (\cite{Fa} and the introduction).

\begin{defi} 
Suppose that $L$ and $l$ are fixed.
An isogeny of  $(\infty,l)$-integral structures
$(V_{\ZZ},s,\rho)$ and $(V'_{\ZZ},s',\rho')$ with $V_{\ZZ}\otimes \QQ = V'_{\ZZ}\otimes\QQ$, 
is a  $\alpha\in \GL(V_{\QQ})$ such that  $\alpha\colon V_{\ZZ}\lto V'_{\ZZ}$
is a  morphism of  $\ZZ$-HS and 
such that 
$\alpha_{l} = \alpha\otimes\ZZ_{l}$ is a morphism of Galois representations.
We say that $\alpha$ is an $l$-isogeny if moreover $\vert V'_{\ZZ}/\alpha V_{\ZZ}\vert=l^n$
for some integer $n$.
\end{defi}

\begin{defi}
Let $(W_{\ZZ},s,\rho)$ be a $(\infty,l)$-integral structure on $V_{\QQ}$ (i.e $W_\QQ = V_\QQ$).
Then $(W_{\QQ}=V_{\QQ},s,\rho\otimes\QQ_{l} )$ is a
$(\infty,l)$-rational  structure on $V_{\QQ}$ called the extension 
of $W_{\ZZ}$ to $\QQ$.

We say that two $(\infty,l)$-integral structures are isogeneous if their extensions to $\QQ$
are isomorphic. Two $(\infty,l)$-integral structures are isogeneous if and only if there exists
an isogeny between them. 
\end{defi}

Note that, by a theorem of Faltings, two abelian varieties $A$ and $B$ over $L$ are isogeneous if and only if the corresponding
$(\infty,l)$-integral structures are isogeneous.

The main objective of this paper is the study of the $(\infty,l)$-integral structures associated to 
$\ol\QQ$-valued points of Shimura varieties and representations of 
appropriate groups.

Let $(G,X)$ be a Shimura datum and $K\subset G(\AAA_{f})$ a neat compact open
subgroup of the form $K = K_{l} \times K^l$ with $K_{l}\subset G(\QQ_{l})$
a compact open subgroup.

Fix a representation $\mu\colon G\lto \GL(V_{\QQ})$ and a lattice 
$V_{\ZZ}\subset V_{\QQ}$ such that $K_{l}\subset \GL(V_{\ZZ_{l}})$.

\begin{lem}
Let $x = [s,\ol{1}]$ be a point in $Sh_{K}(G,X)(L)$.
Fix the point $\wt{x} = [s,1]$  of $Sh_{K^l}(G,X)$ over $x$.
Let $\rho_{\wt{x},l}$ be the $l$-adic representation attached to this data.

Then $(V_{\ZZ}, \mu\circ s, \mu \circ \rho_{\wt{x},l})$ is an 
$(\infty,l)$-integral structure. We call such a structure a structure of Shimura type.
\end{lem}
\begin{proof}
Let $M$ be the Mumford-Tate group of $s$.
Then the Mumford-Tate group of $\mu\circ s$ is $\mu(M)$, therefore
$$
\End_{\ZZ-HS}(V_{\ZZ}) = \End_{\mu(M)}(V_{\ZZ})
$$
where, by abuse of notation, we denote 
$\End_{\mu(M)}(V_{\ZZ}) = \End(V_{\ZZ_l})\cap \End_{\mu(M)}(V_{\QQ_l})$.
On the other hand, the image of  $\mu \circ \rho_{\wt{x},l}$
is contained in $\mu(M)(\ZZ_{l})$ (by Proposition \ref{prop1}), 
the compatibility condition follows.
\end{proof}

We now state the Isogeny and Mumford-Tate conjectures.

\begin{conj}[Isogeny and generalised Mumford-Tate conjectures.]
Let $(V_{\ZZ}, \mu\circ s, \mu \circ \rho_{\wt{x},l})$ be an $(\infty, l)$-integral structure
of Shimura type.
Let $H_{\QQ_{l}}$ be the neutral component of the Zariski closure of the image of  $\rho_{\wt{x},l}$.
The Mumford-Tate conjecture asserts that $H_{\QQ_{l}} = M_{\QQ_{l}}$.

The Isogeny conjecture asserts that every $(\infty,l)$-integral structure of
Shimura type is Tate.
\end{conj}

Notice, that for Shimura varieties of Hodge type, the generalised Mumford-Tate conjecture is actually the 
usual Mumford-Tate conjecture while the Isogeny conjecture is genuinely more general.
In what follows, when referring to the generalised Mumford-Tate conjecture, we will often drop the 
adjective `generalised' which, we hope, will cause no confusion.

\subsection{Isogeny and generalised Mumford-Tate conjectures are equivalent.}

We prove the following.

\begin{teo}
The generalised Mumford-Tate and Isogeny conjectures are equivalent.

More precisely, let $(V_{\ZZ}, s, \rho_{\wt{x},l})$ be the data attached, as before, to a
point $[x, \ol{1}]$ of a Shimura variety $\Sh_{K}(G,X)(L)$ where $L$ is a number field
and $\wt{x}=[s,1]$. 
Let $H_{\QQ_l}$ and $M_{\QQ_l}$ be as before.

Then for any representation $\mu$,
the $(\infty, l)$-structure $(V_{\ZZ}, \mu\circ s, \mu \circ \rho_{\wt{x},l})$ is Tate
if and only if
$$
H_{\QQ_l} = M_{\QQ_l}
$$
\end{teo}
\begin{proof}
Let $(V_{\ZZ},s,\rho)$ be a $(\infty,l)$-integral structure.
Assume the Mumford-Tate conjecture. As $M$ is a reductive group,
the Mumford-Tate conjecture implies that  
neutral component of the Zariski closure of the image of  $\rho_{\wt{x},l}$
is reductive. Therefore $\rho_{\wt{x},l}$ is semi-simple.
As 
$$
\End_{\Gal(\ol L/L)}(V_{\ZZ_l}) = \End_{{\mu \rho(\Gal(\ol L/L))}^{\rm Zar}}(V_{\ZZ_l}) = \End_{\mu(M)}(V_{\ZZ_l}) = \End_{\ZZ-HS}(V_{\ZZ})\otimes\ZZ_l,
$$
the $(\infty,l)$-integral structure $(V_{\ZZ},s,\rho)$ is Tate.

The fact that the Isogeny conjecture  implies the Mumford-Tate conjecture 
will result from the following lemma 

\begin{lem}
Let $H\subset M\subset G$ be the inclusions of connected reductive algebraic groups over
$\QQ_{l}$. 
Suppose that for every finite dimensional representation $\mu\colon G\lto \GL(V_{\QQ_{l}})$,
the following equality of centralisers holds:
$$
Z_{\GL(V_{\QQ_{l}})}(\mu(H)) = Z_{\GL(V_{\QQ_{l}})}(\mu(M))
$$
then $M=H$.
\end{lem}
\begin{proof}
By a theorem of Chevalley, there exists a representation $\mu \colon G_{\QQ_{l}}\lto \GL(V_{\QQ_{l}})$
and a vector $x \in V_{\QQ_{l}}$
such that  $H$ is the stabiliser of the line $\QQ_{l}\cdot x$.

Suppose that $M\not= H$, then $W = M(\QQ_{l})\cdot x$ is an irreducible $M$-module of
dimension $>1$.
Using semi-simplicity of the representations of $M$, we write
$$
V = W \oplus W'
$$   
as $M$-sub-modules.
We note that $W$ and $W'$ are $H$-modules and we 
have a decomposition into $H$-modules (using semi-simplicity of representations of $H$):
$$
V_{\QQ_{l}} = \QQ_{l}\cdot x \oplus W_{1} \oplus W'
$$
A block diagonal matrix with respect to this decomposition is in the centraliser of $H$
but not in the centraliser of $M$ by construction.
This gives a contradiction.
\end{proof}

Let $H$ be the Zariski closure of $\rho(\Gal(\ol L/L))$.
The semi-simplicity of $\rho$ implies that $H$ is a reductive subgroup of $M$.
The generalised Tate conjecture implies that 
$$
Z_{ \GL( V_{\QQ_l} ) } (\mu(H)) = Z_{ \GL( V_{\QQ_l} )} (\mu(M))
$$
for all representations $\mu$.
By applying the lemma to $H\subset M$, we conclude that $H=M$ i.e. the Mumford-Tate conjecture holds.

\end{proof}

\begin{rem}
Note that it is not possible to deduce (at least directly) the Mumford-Tate conjecture from  the Tate isogeny conjecture for abelian varieties proved by Faltings.
The point here being that in order to deduce Mumford-Tate we need the Isogeny conjecture to hold for 
$(\infty, l)$-integral structures attached to  \emph{all} representations of the symplectic group. Faltings only proves it for
the natural symplectic representation.
\end{rem}

\section{Generalised Shafarevich conjecture.}

In this section we state the generalised Shafarevich conjecture
and prove that it is equivalent to the Isogeny conjecture.
One implication is standard and is essentially due to Tate.
For other implication, we show that generalised Mumford-Tate
implies the generalised Shafarevich conjecture. This
involves some non-trivial group-theoretic and measure-theoretic 
arguments.

 We consider a Shimura variety $Sh_{K}(G,X)$, 
 a finite dimensional representation $V_{\QQ}$ of $G$ and a lattice 
 $V_{\ZZ}$ of $V_{\QQ}$
 as in the previous section.

\begin{conj} [generalised Shafarevich conjecture]
Assume that $K=K^lK_{l}$ is a neat compact open subgroup of $G(\AAA_{f})$.
Let $x = [s,\ol{1}]$ be a point of $Sh_{K}(G,X)(L)$ and fix $\wt{x} = [s,1]$
the point of $Sh_{K^l}(G,X)$ above $x$.
Let $(V_{\ZZ}, s, \rho_{\wt{x},l})$ be the corresponding $(\infty,l)$-integral structure
of Shimura type.

The set of $(\infty,l)$-integral structures of Shimura type $l$-isogeneous
to $(V_{\ZZ}, s, \rho_{\wt{x},l})$ forms a finite number of isomorphism classes.
\end{conj}

One notices that in the case where the $(\infty,l)$-integral structures are attached to abelian
varieties, this is (a slightly weaker form of) a statement proved by Faltings.

The main theorem we prove in this section is the following.

\begin{teo}
The generalised Shafarevich conjecture is equivalent to the Isogeny conjecture.

More precisely, let as before $(V_{\ZZ_l}, s ,\rho_{\wt{x},l})$ be an $(\infty,l)$-integral
structure of Shimura type. 
Then the set of $(\infty,l)$-integral structures of Shimura type $l$-isogeneous to $(V_{\ZZ_l}, s ,\rho_{\wt{x},l})$
is finite if and only if for any $\mu$, the structure  $(V_{\ZZ_l}, \mu\circ s ,\mu\circ \rho_{\wt{x},l})$
is Tate.
\end{teo}

\subsection{Generalised Shafarevich implies the Isogeny conjecture.}

In this section we prove that the generalised Shafarevich conjecture implies the 
Isogeny conjecture.
The proofs in this section essentially paraphrase section 2 of \cite{Tate}.

\begin{lem} \label{lem1}
Let $(V_{\ZZ}, s ,\rho)$ be an $(\infty, l)$-integral structure on $V_{\QQ}$.
We suppose that the $l$-isogeny class of $(V_{\ZZ}, s ,\rho)$ contains finitely
many isomorphism classes. For any $\Gal(\ol L/L)$-submodule $V'$ of $V_{\QQ_{l}}$,
there exists a $u \in \End_{\QQ-PHS}(V_{\QQ})\otimes \QQ_{l}$ such that
$$
u(V_{\QQ_{l}})=V'
$$
\end{lem}
\begin{proof}
Let $V'$ be a $\Gal(\ol L/L)$-submodule of $V_{\QQ_{l}}$ and fix a $\QQ_{l}$-sub-module $V''$ of
$V_{\QQ_{l}}$ such that 
$V_{\QQ_{l}}=V'\oplus V''$.

Let $V'_{\ZZ_{l}}:= V'\cap (V_{\ZZ}\otimes \ZZ_{l})$ and $V''_{\ZZ_{l}}:= V''\cap (V_{\ZZ}\otimes \ZZ_{l})$.
These are lattices in $V'$ and $V''$ respectively.
Let
\begin{equation}\label{dec1}
V_{n,\ZZ_{l}} :=  V'_{\ZZ_{l}} + l^n V_{\ZZ_{l}}
\end{equation}

The $\ZZ$-submodule $V_{n,\ZZ} := V_{n,\ZZ_{l}}\cap V_{\ZZ}$ of $V_{\ZZ}$ is free of the same rank
and $V_{n,\ZZ}\otimes \ZZ_{l} = V_{n,\ZZ_{l}}$.
Thus we have constructed a sequence ofl $(\infty, l)$-integral structures $(V_{n,\ZZ}, s, \rho)$.
Note that, by construction, all these structures are $l$-isogeneous to $(V_{\ZZ},s,\rho)$, therefore, by assumption, they fall
 into finitely many isomorphism classes.
Hence, we can assume that all the $(V_{n,\ZZ}, s,\rho)$ are isomorphic to a fixed 
$(V_{n_{0}\ZZ}, s, \rho)$.
We can view each $(V_{n,\ZZ}, s,\rho)$ as a sub-$(\infty, l)$-integral structure of $(V_{n_{0}\ZZ}, s, \rho)$ for $n\geq n_{0}$.

Let $\alpha_n \colon (V_{n_{0}\ZZ}, s, \rho) \lto (V_{n,\ZZ}, s, \rho)$ be an isomorphism and
 $f_n$ be the isogenies given by inclusions between $(V_{n,\ZZ}, s, \rho)$ and  $(V_{n_{0}\ZZ}, s, \rho)$).



Hence $u_n =  f_{n} \alpha_{n}$ is an element of $\End(V_\QQ)$ and it satisfies
$$
u_n(V_{n_0\ZZ_l}) = V_{n,\ZZ_l}\subset V_{n_0\ZZ_l}.
$$

Hence $u_n\otimes \ZZ_{l}$ is in $\End(V_{n_0\ZZ_l})$ which is compact.
After possibly passing to a sub-sequence, we may
assume that $u_n\otimes \ZZ_{l}$ converges to $u \in \End(V_{n_{0}\ZZ_{l}})$.

On the other hand, as $u_n \in \End_{\ZZ-PHS}(V_{n_{0}})\otimes \ZZ_{l}$ which is closed in
$\End(V_{n_{0\ZZ_{l}}})$ (using the fact that the $u_n$ commute to the Mumford-Tate group),
 we see that $u\in  \End_{\ZZ-PHS}(V_{n_{0}})\otimes \ZZ_{l}$.

It is easy to see that there exists $k\in \NN$ such the index of 
$V'_{\ZZ_{l}} + l^n V''_{\ZZ_{l}}$ in $V_{n,\ZZ_{l}} =  V'_{\ZZ_{l}} + l^n V_{\ZZ_{l}}$
is bounded by $l^k$ when $n\rightarrow \infty$.

 Let $v \in V_{n_{0}\ZZ_{l}}$ and let $v = v' + v''$ with $v'\in V'_{\QQ_{l}}$
 and $v''\in V''_{\QQ_{l}}$.  
 Using the previous  discussion we can write
 $u_n v = v'_n + v''_n$ with $l^k v'_{n}\in V'_{\ZZ_{l}}$ and   $v''_{n} \in l^{n-k} V''_{\ZZ_{l}}$.
By making $n$ go to infinity, we see that $u(v) \in V'$. By linearity this implies
that $u(V_{\QQ_{l}})\subset V'$.
To see that $u$ surjects onto $V'$, take any element $v'$ of $V'_{ \ZZ_{l}}$.
As $u_n$ is surjective, there exists an element $v_{n}$ of $V'_{n_{0} \ZZ_{l}}$ such that $u_nv_{n} = v'$.
By compacity we may assume that $v_{n}$ converges to $v\in V'_{n_{0} \ZZ_{l}}$ and 
by passing to the limit, we see that $u(v) = v'$.
By extending scalars to $\QQ_{l}$, we see that $u$ is surjective.
\end{proof}

\begin{prop} \label{lem2}
Let $(V_{\QQ}, s, \rho)$ be a $(\infty, l)$-rational structure on $V_{\QQ}$. 
Suppose that for every $\Gal(\ol L/L)$-submodule $W$ of $V_{\QQ_{l}} = V_{l}$ 
(resp. everery $\Gal(\ol L/L)$-submodule of $(V_{\QQ_l}\times V_{\QQ_l}, s\times s, \rho \times \rho)$),
there exists a $u\in \End_{\QQ-PSH}(V_{\QQ})\otimes\QQ_{l}$
(resp. $u\in \End_{\QQ-PSH}(V_{\QQ}\times V_{\QQ})\otimes\QQ_{l}$
such that $u(V_{l}) = W$  (resp. $u(V_{l}\times V_{l})=W$)    , then
$(V_{\QQ},s,\rho)$ is Tate.
\end{prop}
\begin{proof}
We first verify the semi-simplicity.
Let $W$ be a $\Gal(\ol L/L)$-invariant subspace of $V_{\QQ_l}$. We need to construct a 
$\Gal(\ol L/L)$-invariant complement $W'$.
Consider the right ideal in $\End_{\QQ-PHS}(V_\QQ)\otimes \QQ_l$
$$
I = \{ u\in \End_{\QQ-PHS}(V_\QQ)\otimes \QQ_l : uV_{\QQ_l}\subset W \}
$$
As $I$ is a right ideal in a semi-simple algebra, it is generated by
an idempotent element $e$.
We have $W=eV_{\QQ_l}$ and we let $W' = (1-e)V_{\QQ_l}$, then
$$
V_{\QQ_l} = W\oplus W'
$$
As $\Gal(\ol L/L)$ commute with $\End_{\QQ-PHS}(V_\QQ)\otimes \QQ_l$, the space $W'$ is $\Gal(\ol L/L)$-stable.
This proves the semisimplicity. 

We now prove the second condition.
Let $\alpha$ be an element of $\End(V_l)$ commuting with $\Gal(\ol L/L)$.
Consider the graph of $\alpha$
$$
W = \{ (x, \alpha(x)) : x \in V_l \}
$$
Note that $W$ is a $\Gal(\ol L/L)$-invariant subspace of $V_l \times V_l$, therefore there is a 
$u\in \End_{\QQ-PHS}(V\times V)\otimes \QQ_l$ such that
$$
u (V_l \times V_l) = W
$$
Let  $C$ be the commutant of $\End_{\QQ-PHS}(V_\QQ)\otimes\QQ_l$ in $\End_{\QQ_l}(V_l)$,
and $c$ be an element of $C$. Then 
$
\begin{pmatrix}
  c & 0\\
  0 & c\\
\end{pmatrix}
$
is an element of $\End(V_l \times V_l)$ and it commutes with $\End_{\QQ-PHS}(V\times V)\otimes \QQ_{l}$
and in particular with $u$.
Consequently 
$$
\begin{pmatrix}
  c & 0\\
  0 & c\\
\end{pmatrix}
W = 
\begin{pmatrix}
  c & 0\\
  0 & c\\
\end{pmatrix}
u\mbox{ }     V_l\times V_{l} = u
\begin{pmatrix}
  c & 0\\
  0 & c\\
\end{pmatrix}
V_l\times V_{l} \subset W
$$
Hence, for any $x\in V_l$, 
$$
(cx,c\alpha x) \in W
$$ 
therefore 
$$
\alpha cx = c\alpha x
$$
Thus for all $c\in C$, $\alpha c = c \alpha$ and $\alpha$ belongs to the 
commutant of the commutant of $\End_{\QQ-PHS}(V)\otimes \QQ_l$ in $\End(V_l)$.
By the double commutant theorem, $\alpha$ belongs to $\End_{\QQ-PHS}(V)\otimes \QQ_l$.
It follows that $(V_{\ZZ},s,\rho)$ is Tate.
\end{proof}

Now we can prove that the generalised Shafarevich conjecture implies the generalised Tate conjecture.
Let $(V_\ZZ,s,\rho)$ be an $(\infty, l)$-integral stucture. According to
the generalised Shafarevich Conjecture, the $l$-isogeny classes of 
$(V_\ZZ,s,\rho)$ and $(V_\ZZ \times V_{\ZZ},s\times s,\rho \times \rho)$
fall into finitely many isomorphism classes. 
The lemma \ref{lem1} implies that the assumptions of \ref{lem2} are satisfied, therefore, by
\ref{lem2}, $(V_\ZZ,s,\rho)$ is Tate.

\subsection{Mumford-Tate implies generalised Shafarevich.}

In this section we prove that the Mumford-Tate conjecture implies the generalised Shafarevich conjecture.

\begin{teo}\label{teo5.5}
Let $(V_{\ZZ}, s , \rho)$ be 
an $(\infty , l)$-integral structure such that $U_l:=\rho(\Gal(\ol L/L))$ is an open subgroup of $M(\QQ_{l})$
where $M$ is the Mumford-Tate group of $s$.

The $l$-isogeny class of  $(V_{\ZZ}, s , \rho)$ is a union of finitely many isomorphism classes.
\end{teo}

\begin{lem}
An $(\infty,l)$-integral structure $l$-isogeneous to $(V_{\ZZ},s,\rho)$
admits a representative (in its isomorphism class) of the form
$(W_{\ZZ},s,\rho)$ where $W_{\ZZ}$ is a sublattice of $V_{\ZZ}$ of index a power of $l$
such that $W_{\ZZ_{l}}$ is $\Gal(\ol L/L)$-invariant.
\end{lem}
\begin{proof}
Let $(V'_{\ZZ},s',\rho')$ be an $(\infty,l)$-integral structure $l$-isogeneous to $(V_{\ZZ}, s , \rho)$.
We can find a $\alpha\in  \GL(V_{\QQ})$ such that $s = \alpha s' \alpha^{-1}$
and $\rho = \alpha \rho' \alpha^{-1}$.
The structure  $(V'_{\ZZ},s',\rho')$ is isomorphic to  $(\alpha V'_{\ZZ},s,\rho)$.

For any power $l^n$ of $l$, there is an isomorphism between
 $(\alpha V'_{\ZZ},s,\rho)$ and  $(l^n \alpha V'_{\ZZ},s,\rho)$.
We can choose $n$ such that $l^n \alpha V'_{\ZZ} \subset V_{\ZZ}$ and we let $W_{\ZZ} = l^n \alpha V'_{\ZZ}$.
\end{proof}

Let $(W_{\ZZ}, s, \rho)$ be an $(\infty , l)$-integral structure $l$--isogeneous to $(V_{\ZZ},s,\rho)$
with $W_{\ZZ}\subset V_{\ZZ}$.
We let
$$
T = \{ \alpha U_l \alpha^{-1} \mbox{ where } \alpha \in \GL(V_{\QQ_{l}}), \alpha U_{l} \alpha^{-1} \subset \GL(V_{\ZZ_{l}}) \} 
$$ 
 
\begin{lem}
Let $\alpha\in \GL(V_{\QQ_{l}})$ be such that $\alpha(W_{\ZZ_{l}})=V_{\ZZ_{l}}$.
Then $\alpha U_l \alpha^{-1}\in T$.
\end{lem}  
\begin{proof}
We have the following diagram for any $\sigma \in \Gal(\ol L/L)$:

\[
 \xymatrix{
 W_{\ZZ_l} \ar@{->}[rr]^{\displaystyle \alpha\cdot} \ar@{->}[dd]_{\displaystyle \rho(\sigma)}
      && V_{\ZZ_l} \ar@{->}[dd]^{\displaystyle \alpha \rho(\sigma)\alpha^{-1}}    \\ \\
 W_{\ZZ_l} \ar@{->}[rr]_{\displaystyle \alpha\cdot} && V_{\ZZ_l}
 }
\]

It is clear from this diagram that $\alpha U_l \alpha^{-1}$ stabilises $V_{\ZZ_l}$ and therefore
$$
\alpha U_l \alpha^{-1} \subset \GL(V_{\ZZ_l})
$$
\end{proof}

The group $\GL(V_{\ZZ_l})$ acts on the left on $T$.
We now prove the following proposition.

\begin{prop}\label{prop5.8}
The set $T$ is a finite union of $\GL(V_{\ZZ_l})$-orbits.
\end{prop}
\begin{proof}

Choose a basis for $V_{\ZZ_l}$ and identify $\GL(V_{\ZZ_l})$ with $\GL_n(\ZZ_l)$.
Let $\mu$ be the normalized Haar measure on $U_{l}$ and ${\bf 1}_{\GL_n(\ZZ_l)}$ be
the characteristic function of $\GL_{n}(\ZZ_{l})$.
Let $m\ge 1$ be an integer. We define a function $\psi_{m}$ on $\GL_{n}(\QQ_{l})$
as follows 
$$
\psi_{m}(\alpha):=\int_{U_{l}^m}\prod_{i=1}^m {\bf 1}_{\GL_{n}(\ZZ_{l})}(\alpha t_{i}\alpha^{-1})
d\mu(t_{1})\dots d\mu(t_{m}).
$$

We see that if $\alpha U_{l}\alpha^{-1}\subset \GL_n(\ZZ_l)$, then $\psi_{m}(\alpha)=1$.
The function $\psi_{m}$ is well defined on 
$$
      \GL_{n}(\ZZ_{l}) \backslash \GL_{n}(\QQ_{l})/Z_{\GL_n}(M)(\QQ_{l}).
$$

The following lemma is a variant of a lemma that was communicated to us by L. Clozel.

\begin{lem} If $m$ is large enough, 
the function $\psi_{m}(\alpha)$ goes to $0$ when $\alpha$ goes to $\infty$
in $ \GL_{n}(\QQ_{l})/Z_{\GL_n}(M)  (\QQ_{l})$ (for the usual $l$-adic topology).
\end{lem}
\begin{proof}
Let $m$ be an integer. It suffices to prove that if $m$ is
large enough, then for $(t_1,\dots, t_m)$ outside a subset of $U_l^m$ of measure zero, we have

$$
\prod_{i=1}^m {\bf 1}_{\GL_{n}(\ZZ_{l})}(\alpha t_{i}\alpha^{-1})\longrightarrow 0
$$
when $\alpha\rightarrow \infty$ in $ \GL_{n}(\QQ_{l})/Z_{\GL_n}(M)(\QQ_{l})$.

Let us define
$$
W_{m}=\{(t_{1},\dots,t_{m}), \mbox{ $t_{i}\in U_{l}$, } \mbox{ $t_{i}$ semi-simple and } Z_{GL_{n}(\QQ_{l})}(t_{1},\dots,t_{m})^0
=Z_{\GL_n}(M)^0\}
$$

For $t=(t_{1},\dots,t_{m})\in W_{m}$, we have a morphism
$$
\pi_{t}:   \GL_{n}(\QQ_{l})/Z_{\GL_n}(M)(\QQ_{l}) \rightarrow O(t)
$$
where $O(t):=\{(x t_{1}x^{-1},\dots, x t_{m}x^{-1}), x\in \GL_{n}(\QQ_{l})\}$
is the orbit of $t$ in $\GL_{n}(\QQ_{l})^m$ under the action of
$\GL_{n}(\QQ_{l})$ by conjugation.

The fibres of $\pi_{t}$ are in bijection with 
$ Z_{GL_{n}(\QQ_{l})}(t_{1},\dots,t_{m})/Z_{\GL_n}(M)(\QQ_{l})$
and, by definion of $(t_1,\dots , t_m)$, are finite.

As $t$ is semi-simple,
this orbit is closed and the map $\pi_{t}$ is proper. Therefore,
for $\alpha$ outside a bounded and closed subset of 
$ \GL_n(\QQ_l)/Z_{\GL_n}(M)(\QQ_{l})$, we have

$$
\prod_{i=1}^m {\bf 1}_{\GL_n(\ZZ_{l})}(\alpha t_{i}\alpha^{-1})=0.
$$

Therefore, the lemma is a consequence of the following statement.

\begin{lem}\label{lem4}
For $m$ large enough, the set $W_{m}$ has measure $1$ in $U_l^m$.
\end{lem}

We start by showing the following.

\begin{lem}\label{lem5}
Let $H$ be a $\QQ_l$-algebraic subgroup of $\GL_{n,\QQ_{l}}$. 
Assume that $Z_{\GL_n}(M)\subset H$ and that $Z_{\GL_n}(M)^0\neq H^0$. 

Then the set
$\{t\in U_l, \ \ H^0\subset Z_{\GL_{n,\QQ_{l}}} (t)\}$ has $\mu$-mesure $0$.
\end{lem}

\begin{proof} 
Let $\Theta\subset U_{l}$ be a subset of positive measure.
Let $\Theta'$ be the subgroup of $U_{l}$ generated by $\Theta$ and
$\Theta''$ the closure of $\Theta'$ for the analytic topology. 
Then $\Theta''$ is a compact subgroup of $U_{l}$ of positive measure.
Then $\Theta''$ is an open compact subgroup of $U_l$.  

An element $g\in \GL_n(\QQ_l)$
that centralizes $\Theta$, also centralizes $\Theta''$. 
Hence we need to prove that if $U'_l \subset M(\QQ_l)$ is a compact open subgroup,
then $Z_{\GL_{n,\QQ_l}}(U'_l)^0 = Z_{\GL_n}(M)^0$.

We first remark that  $Z_{\GL_{n,\QQ_l}}(U'_l)$ is contained in $N_M^0$ 
(neutral componant of the normaliser of $M$). If $s\in  Z_{\GL_{n,\QQ_l}}(U'_l)$
but $s\notin N_{M}$, then $M_{s}:=M\cap sMs^{-1}$ is a subgroup of $M$
such that $\mbox{dim}(M_{s})<\mbox{dim}(M)$ (notice that $M$ is connected).
But $M_{s}(\QQ_{l})$ contains the open compact subgroup $U'_{l}$ of $M(\QQ_{l})$.
One can check that this gives a contradiction.

We have the following decomposition as almost direct products:
$$
N_M^0 =M^{der}Z_{\GL_n}(M).
$$
We have the following exact sequence
$$
1 \lto Z_{\GL_n}(M)^0(\QQ_l) \lto N_M(^0\QQ_l) \lto M^{\rm ad}(\QQ_l)
$$
An element of $Z_{\GL_{n,\QQ_l}}(U'_l)$ goes to the centraliser of a compact open subgroup
$M^{\rm ad}(\QQ_l)$ and is therefore trivial.
This shows that $Z_{\GL_{n,\QQ_l}}(U'_l) = Z_{\GL_n}(M)$.
\end{proof}

We define the subset $\Theta_{m}$ of $U_{l}^m$ to be the 
set of the $(t_{1},\dots,t_{m})\in U_{l}^m$ such that 
the first coordinate $t_{1}$ is an arbitrary element in $U_{l}$.

If $Z_{\GL_{n,\QQ_{l}}}(t_{1})^0=Z_{\GL_n}(M)^0$, then $t_{2}$ is an arbitrary element of $U_{l}$.
If $Z_{\GL_{n,\QQ_{l}}}(t_{1})^0\neq Z_{\GL_n}(M)^0$ we require that
$Z_{\GL_{n,\QQ_l}}(t_{1})^0$ is not contained in $Z_{\GL_{n,\QQ_l}}(t_{2})^0$. 
By the above lemma $t_{2}$ is in a subset of $U_l$ of $\mu$-mesure $1$.
On the other hand
$$
Z_{\GL_{n,\QQ_l}}(t_{1},t_{2})^0 \subsetneqq Z_{\GL_{n,\QQ_l}}(t_{1})^0. 
$$
By repeating the process we construct the required set $\Theta_m$ as follows.
Let $(t_{1},\dots,t_{r})$
be the first $r$ coordinates of an element of $\Theta_{m}$. 
If $Z_{\GL_{n,\QQ_l}}(t_{1},\dots,t_{r})^0=Z_{\GL_n}(M)^0$, 
 $t_{r+1}$ is arbitrary in $U_{l}$. 
If $Z_{\GL_{n,\QQ_l}}(t_{1},\dots,t_{r})^0\neq Z_{\GL_n}(M)^0$, 
we impose that 
$Z_{\GL_{n,\QQ_l}}(t_{1},\dots,t_{r})^0$ is not in $Z_{\GL_{n,\QQ_l}}(t_{r+1})^0$.
The element $t_{r+1}$ is in a subset of $U_l$ of $\mu$-mesure $1$.

By construction of $\Theta_{m}$, for all $m$,
$\Theta_{m}$ is of mesure $1$ and for $m$ large enough
the elements of $\Theta_{m}$ satisfy 
$$
Z_{\GL_{n,\QQ_l}}(t_{1},t_{2},\dots,t_{m})^0=Z_{\GL_n}(M)^0
$$

As $M$ is reductive and connected, the set of semi-simple elements
is open in $M$ and the set $U_l^{ss}$ of semi-simple elements of
$U_l$ is of measure $1$. We deduce that for $m$ large enough,
$\Theta_{m}\cap (U_{l}^{ss})^m$ is of measure $1$ and is
contained in $W_m$. This finishes the proof of the lemma.
\end{proof}

Let $m$ be large enough so that the conclusion of the above lemma holds.
The set
$$
A:=\{\alpha\in \GL_{n}(\QQ_{l}),\   \alpha U_{l}\alpha^{-1}\subset \GL_{n}(\ZZ_{l})\}
$$
has a bounded image in $\GL_{n}(\QQ_{l})/   Z_{\GL_n}(M)(\QQ_{l})$
(note that ${\psi_{m}}_{\vert A}=1$). We deduce that 
$$
\GL_{n}(\ZZ_{l})\backslash A/  Z_{\GL_n}(M)(\QQ_{l})
$$
is a finite set. There is a surjection
$ A/ Z_{\GL_n}(M)(\QQ_{l})    \rightarrow T$. The finiteness
of $\GL_{n}(\ZZ_l)\backslash T$ follows.
\end{proof}

Let $\{ \theta_{1} U_{l} \theta_{1}^{-1}, \dots ,  \theta_{r} U_{l} \theta_{r}^{-1}\}$ be a set of representatives for $\GL(V_{\ZZ_{l}})$-orbits in $T$.

Let $\alpha \in \GL(V_{\QQ_{l}})$ such that $\alpha W_{\ZZ_{l}} = V_{\ZZ_{l}}$.
Multiplying $\alpha$ on the left by an element of $\GL(V_{\ZZ_{l}})$ does not affect this condition, therefore we may assume that
there exists a $i\in \{ 1,\dots , r\}$ such that 
$$
\alpha U_{l} \alpha^{-1} = \theta_{i} U_{l} \theta_{i}^{-1}
$$

In what follows we study the set of $(\infty, l)$-integral structures $(W_{\ZZ},s,\rho)$, $l$-isogeneous to $(V_{\ZZ}, s, \rho)$
with $W_{\ZZ}\subset V_{\ZZ}$ such that the set
$$
T_{\theta} = \{ \alpha \in \GL(V_{\QQ_{l}}), \alpha W_{\ZZ_{l}} =V_{\ZZ_{l}} \mbox{ and }   \alpha U_{l} \alpha^{-1} = \theta U_{l} \theta^{-1}  \}
$$
is non-empty (for a fixed $\theta\in \GL(V_{\ZZ_{l}})$ among the set of $\theta_i$).

Let us define the set 
$$
S = \{ \alpha \in \GL(V_{\QQ_{l}}), \alpha W_{\ZZ_{l}} =V_{\ZZ_{l}} \mbox{ and }   \alpha U_{l} \alpha^{-1} = U_{l}  \}
$$

\begin{lem}
Let $Z := Z_{\GL(V_{\QQ_{l}})}(M)$.
The set
$$
U_{l}\backslash S / Z(\QQ_{l})
$$
is finite.
\end{lem}
\begin{proof}

Consider
$$
S' = \{ \alpha\in \GL_n(\QQ_l), \alpha U_l \alpha^{-1} = U_l \}
$$
Clearly, it suffices to prove that
$$
U_l \backslash S' / Z(\QQ_l)
$$
is finite.
Note that $S'$ is contained in 
$N_{\GL_n}(M)(\QQ_l)$ (the normaliser of $M$).
As we are only interested in finiteness, we just have to prove that 
$$
U_l \backslash S'' / Z(\QQ_l)
$$
for  the subset $S''=S'\cap  N_{\GL_n}(M)^0(\QQ_l)$ of $S$.
Write 
$$
N = N_{\GL_n}(M)^0 = M Z=M^{der}Z'
$$
We have an exact sequence
$$
1 \lto Z'(\QQ_l) \lto N(\QQ_l) \lto M^\ad(\QQ_l)
$$
Let $U_{l}^{ad}$ be the image of $U_{l}$ in  $M^{ad}(\QQ_{l})$.
The elements of $S''$ go to the normaliser of $U_l^\ad$ in $M^\ad(\QQ_l)$.
The normaliser of a compact open subgroup  the $\QQ_{l}$-points of
a semi-simple group  is a compact open subgroup.
It follows that
$$
S'' = Z(\QQ_l) V_l
$$
where $V_l$ is a compact open subgroup of $M(\QQ_l)$ containing $U_l$.
The result follows.
\end{proof}

\bigskip

Let us write the finite decomposition
$$
S = \coprod_{i} U_{l} t_{i} Z(\QQ_{l})
$$
We fix $i$ and let $t = t_{i}$.
Note that there is a bijection
$$
T_{\theta}\lto S \  \  \  \alpha\mapsto \theta^{-1}\alpha
$$
Thus, we may assume that the $(\infty,l)$-integral structures $(W_{\ZZ},s,\rho)$ $l$-isogeneous to $(V_{\ZZ}, s, \rho)$
are equipped with an $\alpha \in T_{\theta}$ such that
$$
\alpha W_{\ZZ_{l}} = V_{\ZZ_{l}} \mbox{ and } \theta^{-1}\alpha = u t z
$$
where $u\in U_l$ and $z \in Z(\QQ_l)$.
We rewrite this as
$$
\theta^{-1} (\theta u^{-1} \theta^{-1}) = t z
$$
Note that the element $\theta u^{-1} \theta^{-1}$ is in $\alpha U_{l} \alpha^{-1}$ and 
fixes $\alpha W_{\ZZ_{l}}$.
Therefore, we can assume that there is an $\alpha \in T_{\theta}$ such that
$$
\alpha W_{\ZZ} = V_{\ZZ} \mbox{ and } \theta^{-1} \alpha = t z
$$

\bigskip 

We can now start the proof of the theorem \ref{teo5.5}.
Consider a sequence $(W_{n,\ZZ}, s,\rho)$ of  $(\infty, l)$-integral structures isogeneous to $(V_{\ZZ}, s, \rho)$. 
Recall that our aim is to prove that this sequence is finite up to isomorphism.

By what preceeds, we see that without loss of generality, we can assume that there
 exists an $\alpha_{n}\in \GL(V_{\QQ_{l}})$ such that
$$
\alpha_{n} W_{n,\ZZ_{l}} = V_{\ZZ_{l}} \mbox{ and } \theta^{-1} \alpha_{n} = t z_{n}
$$
for some $z_n \in Z(\QQ_{l})$.
For every pair $n, n'$, we let 
$$
\beta_{n,n'} := \alpha_{n'}^{-1}\alpha_n = z_{n'}^{-1} z_n \in Z(\QQ_l)
$$
We have $\beta_{n,n'} W_{n,\ZZ_{l}} = W_{n',\ZZ_{l}}$.
Note that the fact that $\beta_{n',n}\in Z(\QQ_l)$ 
shows that $\beta_{n,n'}$ induces an isomorphism of Galois representations
between $(W_{n,\ZZ_l}, \rho)$ and $(W_{n',\ZZ_l}, \rho)$.

Fix $n_0$ and for every $n$, consider 
$$
\beta_n := \beta_{n,n_{0}}
$$
so that $\beta_{n} W_{n,\ZZ_{l}} = W_{n_0,\ZZ_{l}}$.

Consider the open compact subgroup of $Z(\QQ_{l})$ defined by
$$
H_{0} = \GL(W_{n_{0},\ZZ_{l}}) \cap Z(\QQ_{l}).
$$
We have a finite decomposition
$$
Z(\QQ_{l}) = \coprod_{i} H_{0} \gamma_{i} Z(\QQ)
$$
for some elements $\gamma_{i}\in Z(\QQ_{l})$.
Without loss of generality we can assume that for all $n$,
$$
\beta_{n} = h_{n} \gamma z_{n}
$$
where $\gamma \in Z(\QQ_{l})$ is fixed and $h_{n}\in H_{0}, z_n \in Z(\QQ)$.
As 
$$
H_0 = \GL(W_{n_0,\ZZ_l})\cap Z(\QQ_l),
$$
we see that we can replace $\beta_n$ with $h_n^{-1}\beta_n$
and therefore assume that
$$
\beta_n = \gamma z_n
$$
where $z_n \in Z(\QQ)$.

We find that for any pair $(n,n')$,
$$
\lambda_{n,n'} = \beta_{n'}^{-1}\beta_n
$$
is an element of $Z(\QQ)$ and satisfies 
$$
\lambda_{n,n'} W_{n,\ZZ_l} = W_{n',\ZZ_l}
$$

As $\lambda_{n,n'}$ is in $Z(\QQ) \cap Z(\QQ_l)$ (intersection inside $Z(\AAA_f)$), we see that $\lambda_{n,n'}(W_{n,\ZZ}) = W_{n',\ZZ}$, therefore $\lambda_{n,n'}$
is an isomorphism of  $(\infty, l)$-integral structures.

\section*{Appendix.}

In this section we prove the following theorem which
implies that making the assumption SV5 of \cite{Milne}
does not cause any loss of generality.
The proof is extracted from the unpublished notes by Moonen \cite{Momo}.

\begin{lem}
Let $(G,X)$ be a Shimura datum such that $G$ is the generic Mumford-Tate
group on $X$. Then $Z_{G}(\QQ)$ is discrete in $Z_{G}(\AAA_f)$.
\end{lem}
\begin{proof}
By assumption, $G$ is the Mumford-Tate group of 
a point $h$ of $X$.
Choose a faithful rational representation $G\hookrightarrow GL(V)$ of $G$.
Composed with $h$, it gives a pure rational Hodge structure on $V$
and and its Mumford-Tate group $M$ is a subgroup of $\GL(V)$.

Let $\UUU_1$ be the subgroup of $\SSS$ consisting of elements of norm one.
We define the Hodge group $H:=Hg(h)$ to be the 
smallest subgroup  of $\GL(V)$ such that 
the restriction of $h$ to $\UUU_{1}$ factorises through 
$H_\RR$.
The Mumford-Tate group is the semi-direct product of $\GG_{m\QQ}$ and 
$Hg(h)$ if the weight is not zero and $M= Hg(h)$ if the weight is zero.
Let $C = h(i)$, of course, $C$ is contained in $Hg(h)(\RR)$ and $\sigma := ad(C)$ defines an 
involution of $Hg(h)_{\RR}$ (note that $C^2 = -id$) and we have an inner form
$Hg(h)^\sigma$ which has the property that $Hg(h)^{\sigma}(\RR)$ is compact
and $Hg(h)$ is a reductive $\QQ$-group. 
We write
$$
Hg(h) = Z_H^0 H_1 \cdots H_r
$$ 
where the $H_i$ are the almost simple factors of $Hg(h)^{der}$.

We note that for each factor in this decomposition (including $Z_H^0$), $\sigma = ad(C)$ defines a compact form over $\RR$.
On the other hand we have
$(Z_H^0)^{\sigma}(\RR) = Z_H^0(\RR)$ (because $(Z_H^0)^{\sigma}(\RR) = \{x\in Z_H(\CC) : \ol{x} = C^{-1} x C = x\} = Z_H^0(\RR)$),
therefore $Z_H^0(\RR)$ is compact and hence $Z_H^0(\QQ)$ is discrete in $Z_H^0(\AAA_f)$.
It is also true that $\GG_{m}(\QQ)$ is discrete in $\GG_m(\AAA_f)$. We conclude that $M(\QQ)$ is discrete in
$M(\AAA_f)$.  
\end{proof}

\end{document}